\documentclass[11pt]{article}
\usepackage{amsthm}
\usepackage{amsmath,a4wide}
\usepackage{amssymb,color}

\theoremstyle{theorem}
\newtheorem{theorem}{Theorem}

\theoremstyle{definition}

\newtheorem{remark}{Remark}

\newtheorem{lemma}{Lemma}
\newtheorem{coro}{Corollary}

\def\ni{\noindent}

\begin{document}
\title{Binary Compositions and Semi-Pell Compositions}
\date{}
\maketitle

\begin{center}
William J. Keith$^{1}$, Augustine O. Munagi$^{2}$\\
$^{1}$Department of Mathematical Sciences, Michigan Technological University \\
$^{2}$School of Mathematics, University of the Witwatersrand, Johannesburg\\
$^{1}$wjkeith@mtu.edu, $^{2}$augustine.munagi@wits.ac.za
\end{center}

\begin{abstract} In analogy with the semi-Fibonacci partitions studied recently by Andrews, we define semi-Pell compositions and semi-$m$-Pell compositions.  We find that these are in bijection with certain weakly unimodal $m$-ary compositions.  We give generating functions, bijective proofs, and a number of unexpected congruences for these objects.
\end{abstract}

\section{Introduction}\label{introd0}
A composition of a positive integer $n$ is an ordered partition of $n$, that is, any sequence of positive integers $(n_1,\ldots,n_k)$ such that $n_1 +\ldots + n_k =n$.
Compositions of $n$ will be represented as vectors with positive-integer entries.

Inspired by a recent paper of Andrews on semi-Fibonacci partitions \cite{An1}, we study the set $SP(n)$ of semi-Pell compositions, defined as follows: 

$SP(1)=\{(1)\},\,  SP(2)=\{(2)\}$.\\ 
If $n>2$ and $n$ is even then

$SP(n)=\{C\mid C\ \mbox{is a semi-Pell composition of $\frac{n}{2}$ with each part doubled}\}$.\\
If $n$ is odd, then a member of $SP(n)$ is obtained by inserting 1 at the beginning or the end of each composition in $SP(n-1)$, and by adding 2 to the single odd part in a composition in $SP(n-2)$. (Indeed it follows by induction that every semi-Pell composition with odd weight contains exactly one part odd part).\\
As an illustration we have the following sets for small $n$:

$SP(1)=\{(1)\}$

$SP(2)=\{(2)\}$

$SP(3)=\{(1,2),(2,1),(3)\}$

$SP(4)=\{(4)\}$

$SP(5)=\{(1,4),(4,1),(3,2),(2,3),(5)\}$

$SP(6)=\{(2,4),(4,2),(6)\}$

$SP(7)=\{(1,2,4),(2,4,1),(1,4,2),(4,2,1),(1,6),(6,1),(3,4),(4,3),(5,2),(2,5),(7)\}$

$\ldots{} \ldots{} \ldots{} \ldots{} $\\
Thus if we define $sp(n)=|SP(n)|$, we obtain that 

$sp(1)=sp(2)=1,\, sp(3)=3,\, sp(4)=1,\, sp(5)=5,\, sp(6)=3,\, sp(7)=11,\, sp(8)=1,$ 

$\ sp(9)=13,\ sp(10)=5, \ldots{}$.\\
Hence we see that $sp(n)=0$ if $n<0$ and $sp(0)=sp(1)=1$, and for $n>1$, the following recurrence holds:
\begin{equation}\label{eq00}
sp(n) = \begin{cases}
sp(n/2)& \mbox{if $n$ is even},\\
2\cdot sp(n-1)+sp(n-2)& \mbox{if $n$ is odd}.
\end{cases}
\end{equation}

The semi-Pell sequence $\{sp(n)\}_{n>0}$ occurs as sequence number A129095 in the Online Encyclopedia of Integer Sequences \cite{Sl}. However, there seems to be no connection of the sequence with compositions until now. The companion sequence A129096 records the fact that the bisection of the semi-Pell sequence $sp(2n-1),\, n>0$ is monotonically increasing: $sp(2n+3)=2sp(2n+2)+sp(2n+1)>sp(2n+1)$ for all $n\geq 0$.

A weakly unimodal composition (or stack) is defined to be any composition of the form $(a_1,a_2,\ldots,a_r,c,b_s,\ldots,b_1)$ such that
$$1\leq a_1\leq a_2\leq \cdots\leq a_r\leq c>b_s\geq \cdots\geq b_2\geq b_1.$$
The set of the $a_i$ or the $b_j$ may be empty. The study of these compositions was pioneered by Auluck \cite{Au} and Wright \cite{Wr} and continues to instigate research (see, for example \cite{An1,An2,Fr}). The ``concave" compositions studied by Andrews in \cite{An2} are weakly unimodal  compositions with unique largest parts.

We will associate the set of semi-Pell compositions with a class of restricted unimodal compositions into powers of 2. 

Binary compositions are compositions into powers of 2. Let $OC(n)$ be the set of weakly unimodal binary compositions of $n$ such that each part size occurs together, or ``in one place,'' an odd number of times. In other words every part size lies in a distinct `colony' (see also Munagi-Sellers \cite{MS}). For example members of $OC(45)$ include, using the frequency notation, $(16, 4^3, 2^3, 1^{11}),(2^5, 4, 8^3, 1^7), (1^7, 8, 16, 2^7)$, but the following weakly unimodal binary compositions of 45 do not belong to $OC(45)$: $(2^3,4^3,16,2,1^{9}), (1^3,2^5, 4, 8^3, 1^4)$.  
\medskip

\begin{lemma}\label{lem2} The enumeration function $oc(n)$ satisfies the recurrence:
\begin{equation}\label{eq01}
oc(n) = \begin{cases}
oc(n/2)& \mbox{if $n$ is even},\\
2\cdot oc(n-1)+oc(n-2)& \mbox{if $n$ is odd}.
\end{cases}
\end{equation}
\end{lemma}
\begin{proof}
The recurrence is obtained in a similar manner to that of $sp(n)$.

$OC(1)=\{(1)\},\, OC(2)=\{(2)\}$.

If $n$ is even and $n>2$, then $OC(n)$ is obtained by doubling each member of $OC(n/2)$.

If $n$ is odd, then $OC(n)$ is obtained from the union of the two sets: (i) set of compositions obtained by inserting 1 before or after each composition in $OC(n-1)$; and (ii) set of compositions obtained by inserting two 1's to the single cluster of 1's in each composition in $OC(n-2)$. Thus, for example we have

$OC(1): (1)$

$OC(2): (2)$

$OC(3): (1,2),(2,1),(1^3)$

$OC(4): (4)$

$OC(5): (1,4),(4,1),(1^3,2),(2,1^3),(1^5)$

$\ldots {}\ldots {}\ldots {}\ldots {}$\\
The result follows.
\end{proof}

The equality of initial conditions and recurrences immediately gives the equality of $sp(n)$ and $oc(n)$ and raises the natural questions of a bijective proof, and of their common generating function.

\begin{theorem}\label{thm2} For integers $n\geq 0$,
\begin{equation}\label{eq2}
sp(n) = oc(n).
\end{equation}
Their common generating function is $$\sum_{n=0}^\infty sp(n) x^n = 1+ \sum_{i=0}^\infty \frac{x^{2^i}}{1-x^{2^{i+1}}}\prod\limits_{t=0}^{i-1}\left(1+\frac{2x^{2^t}}{1-x^{2^{t+1}}}\right)= \sum_{n=0}^\infty oc(n)x^n.$$
\end{theorem}
\begin{proof} We first prove the generating function claim and then give the bijective proof.

\noindent First Proof (generating functions): Let $P(x)=\sum\limits_{n\geq 0}sp(n)x^n$. Then
\begin{align*}
P(x)&=\sum_{n\geq 0}sp(2n)x^{2n} + \sum_{n\geq 0}sp(2n+1)x^{2n+1}\\
&=\sum_{n\geq 0}sp(2n)x^{2n} + 2\sum_{n\geq 1}sp(2n)x^{2n+1} + \sum_{n\geq 1}sp(2n-1)x^{2n+1} + x\\
&=\sum_{n\geq 0}sp(n)x^{2n} + 2\sum_{n\geq 1}sp(n)x^{2n+1} + \sum_{n\geq 0}sp(2n+1)x^{2n+3} + x\\
&=P(x^2) + 2x(P(x^2)-1) + x + x^2\sum_{n\geq 0}sp(2n+1)x^{2n+1}.
\end{align*}
Eliminating the last sum by the first equality, 
$$P(x) =P(x^2) + 2xP(x^2) - x + x^2(P(x)-P(x^2))$$
$$\implies P(x) + \frac{x}{1-x^2} = \frac{1+2x-x^2}{1-x^2}P(x^2).$$
To iterate the last equation, we begin by denoting both sides by $P_1(x)$ and obtain
$$P_1(x) = \frac{1+2x-x^2}{1-x^2}\left(P_1(x^2)- \frac{x^2}{1-x^4}\right),$$
which gives 
$$P_1(x)+\frac{(1+2x-x^2)x^2}{(1-x^2)(1-x^4)} = \frac{1+2x-x^2}{1-x^2}P_1(x^2).$$
Denoting both sides of the last equation by $P_2(x)$ we obtain
$$P_2(x) = \frac{1+2x-x^2}{1-x^2}\left(P_2(x^2)- \frac{(1+2x^2-x^4)x^4}{(1-x^4)(1-x^8)} \right),$$
which gives
$$P_2(x) + \frac{(1+2x-x^2)(1+2x^2-x^4)x^4}{(1-x^2)(1-x^4)(1-x^8)} = \frac{1+2x-x^2}{1-x^2}P_2(x^2).$$
Similarly,
$$P_3(x) + \frac{(1+2x-x^2)(1+2x^2-x^4)(1+2x^4-x^8)x^8}{(1-x^2)(1-x^4)(1-x^8)(1-x^{16})} = \frac{1+2x-x^2}{1-x^2}P_3(x^2).$$
In general we have, 
$$P_i(x) + \frac{x^{2^i}}{1-x^{2^{i+1}}}\prod\limits_{r=0}^{i-1}\frac{(1+2x^{2^r}-x^{2^{r+1}})}{(1-x^{2^{r+1}})} = \frac{1+2x-x^2}{1-x^2}P_i(x^2),\ i>0.$$

This suggests that $P(x)$ is given by the limiting value of $P_i(x)$ as $i$ tends to infinity, plus a constant 1 to suit our initial conditions:

$$P(x) = 1+ \lim\limits_{i\rightarrow\infty} P_i(x) = 1+ \sum_{i=0}^\infty \frac{x^{2^i}}{1-x^{2^{i+1}}}\prod\limits_{t=0}^{i-1}\left(1+\frac{2x^{2^t}}{1-x^{2^{t+1}}}\right)= \sum_{n=0}^\infty oc(n)x^n.$$

That this is the generating function for $oc(n)$ is immediate, as the $i$ term counts $OC$ compositions with largest part $2^i$. 

Additionally, it may be verified algebraically that the generating function $P(x)$ satisfies the functional equation  
$$P(x) + \frac{x}{1-x^2} = \frac{1+2x-x^2}{1-x^2}P(x^2)$$ 
\noindent that was constructed for $sp(n)$ as well.

\vskip 5pt

\ni Second Proof (combinatorial bijection). Each part $t$ of $\lambda\in SP(n)$ can be expressed as $t=2^i\cdot h,\, i\geq 0$, where $h$ is odd. Now transform $t$ as follows: 

if $i>0$, i.e., $t$ is even, then $t=2^i\cdot h\longmapsto 2^i,2^i,\ldots,2^i (h\ \mbox{times})$;

if $i=0$, i.e., $t$ is odd and occurs as a first or last part of $\lambda$, then $t=h\longmapsto 1,1,\ldots,1 (h\ \mbox{times})$.

This gives a unique binary composition in $OC(n)$ provided we retain the cluster of $2^i$'s or 1's corresponding to each $t$ in consecutive positions, and we do not re-order the parts of the resulting binary composition.
(In other words, each (possibly repeated) part-size appears in exactly one place in the image).

For the inverse map we simply write each $\beta\in OC(n)$ in the one-place exponent notation, by replacing every $r$ consecutive equal parts $x$ with $x^r$, to get $\beta=(\beta_1^{u_1},\ldots,\beta_s^{u_s})$, with the $u_i$ odd and positive, and containing at most one instance of a 1-cluster which may be $\beta_1^{u_1}$ or $\beta_s^{u_s}$. Since each $\beta_i^{u_i}$ has the form $(2^{j_i})^{u_i},\, j_i\geq 0$, we apply the transformation:
$$\beta_i^{u_i}=(2^{j_i})^{u_i}\longmapsto 2^{j_i}u_i.$$
This gives a unique composition in $SP(n)$ provided that the resulting parts retain their relative positions. Indeed the image may contain at most one odd part which occurs precisely when $j_i=0$.
\end{proof}

As an illustration of this bijection note that the 13 members of $SP(9)$ and the 13 members of $OC(9)$ listed below correspond one-to-one under the bijection.

\ni Members of $SP(9)$:

$(1,8),(8,1),(3,2,4),(2,4,3),(3,4,2),(4,2,3),(3,6),(6,3),(5,4),(4,5),(7,2),(2,7),(9)$\\
\ni Members of $OC(9)$:

$(1, 8), (8, 1), (1^3, 2, 4), (2, 4, 1^3), (1^3, 4, 2), (4, 2, 1^3),(1^3, 2^3),(2^3, 1^3), (1^5, 4), (4, 1^5),(1^7, 2),$ 

$(2, 1^7), (1^9)$.
\bigskip

The following result is easily deduced from the first part of the relation \eqref{eq00}.
\begin{coro}\label{coro1} Given any positive nonnegative integer $v$, then
 $$sp(2^j(2v+1))=sp(2v+1)\ \forall\ j\geq 0.$$

In particular $\, sp(2^j)=1\ \forall\ j\geq 0.$
\end{coro}
\medskip

Parities of $sp(n)$ and $n$ agree at odd values of $n$ modulo 4.
\begin{theorem}\label{congm2} Given any nonnegative integer $n$, then
$$sp(2n+1)\equiv 2n+1\pmod{4}.$$
\end{theorem}
\ni This theorem implies that

$sp(4n+1)\equiv 1\pmod{4},\ n\geq 0$\\
and
 
$sp(4n+3)\equiv 3\pmod{4},\ n\geq 0$.
\vskip 5pt

\ni The proof requires the following lemma which may be established by an easy induction argument, or by observing that any $n$ has exactly one OC composition into exactly one size of power of 2, and all other OC compositions may be paired by whether they have their smallest part on the left or right.

\begin{lemma}\label{lemm2} $sp(n)$ is odd for all integers $n\geq 0$.
\end{lemma}

\begin{proof}[Proof of Theorem \ref{congm2}] By induction on $n$. Note that $sp(1)=1$ and $sp(3)=|\{(1,2),(2,1),(3)\}|=3$. So the assertion holds for $n=0,1$. Assume that $sp(2j+1)\equiv 2j+1\pmod{4}$ for all $j<n$. Then 

$sp(2n+1) = 2sp(2n) + sp(2n-1) = 2sp(n) + sp(2n-1)$.\\
Then by Lemma \ref{lemm2} $sp(n)$ is odd, say $2u+1$, and the inductive hypothesis shows that $sp(2n-1)\equiv 2n-1\pmod{4}$. Hence

$sp(2n+1) = 2(2u+1) + 2n-1+4t = 4(u+t) + 2n+1\equiv 2n+1\pmod{4}$. 
\end{proof}

Andrews in \cite{An1} denotes by $ob(n)$ the number of partitions of $n$ into powers of 2, in which each part size appears an odd number of times.  Theorem \ref{congm2} has the corollary that, for $n \equiv
2i+1 \pmod{4}$, the number of these in which exactly 2 part sizes
appear is congruent to $i$ mod 2, for each of these correspond to exactly two compositions enumerated by $oc(n)$ by reordering, while every $n$ has 1 additional such partition (and composition) into exactly 1 part size, and those into three or more part sizes correspond to a multiple of four such compositions.

In Section \ref{general} we obtain extensions of semi-Pell compositions to semi-$m$-Pell compositions and establish the corresponding extension of Theorem \ref{thm2}. We also give an alternative characterization of semi-$m$-Pell compositions in Section \ref{struct}. Then in Section \ref{pties} we prove some congruences satisfied by the enumeration function of these compositions.

\section{The Semi-$m$-Pell Compositions}\label{general}
We generalize the set of semi-Pell compositions to the set $SP(n,m)$ of semi-$m$-Pell compositions as follows:

$SP(n,m)=\{(n)\},\ n=1,2,\ldots,m$.\\
If $n>m$ and $n$ is a multiple of $m$, then 

$SP(n,m)=\{\lambda\mid \lambda\ \mbox{is a composition of $\frac{n}{m}$ with each part multiplied by}\ m\}$.\\ 
If $n$ is not a multiple of $m$, that is, $n\equiv r\pmod{m},\, 1\leq r\leq m-1$, then $SP(n,m)$ arises from two sources: first, compositions obtained by inserting $r$ at the beginning or at the end of each composition in $SP(n-r,m)$, and second, compositions obtained by adding $m$ to the single part of each composition $\lambda \in SP(n-m,m)$ which is congruent to $r\pmod{m}$. (Note that $\lambda$ contains exactly one part which is congruent to $r$ modulo $m$, see Lemma \ref{lem001} below).

\begin{lemma}\label{lem001}
Let $\lambda\in SP(n,m)$. 

If $m\mid n$, then every part of $\lambda$ is a multiple of $m$. 

If $n\equiv r\pmod{m},\, 1\leq r <m$, then $\lambda$ contains exactly one part $\equiv r\pmod{m}$.
\end{lemma}
\begin{proof}
If $m\mid n$, the parts of a composition in $SP(n,n)$ are clearly divisible by $m$ by construction. 

For induction note that $SP(r,m)=\{(r)\},\, r=1,\ldots,m-1$, so the assertion holds trivially. Assume that the assertion holds for the compositions of all integers $<n$ and consider $\lambda\in SP(n,m)$ with $1\leq r<m$. Then $\lambda$ may be obtained by inserting $r$ at the beginning or end of a composition $\alpha\in SP(n-r,m)$. Since $\alpha$ consists of multiples of $m$ (as $m|(n-r)$), $\lambda$ contains exactly one part $\equiv r\pmod{m}$. Alternatively $\lambda$ is obtained by adding $m$ to the single part of a composition $\beta\in SP(n-m,m)$ which is $\equiv r\pmod{m}$. Indeed $\beta$ contains exactly one such part by the inductive hypothesis. Hence the assertion is proved.
\end{proof}

\ni As an illustration we have the following sets for small $n$ when $m=3$:

$SP(1,3)=\{(1)\}$

$SP(2,3)=\{(2)\}$

$SP(3,3)=\{(3)\}$

$SP(4,3)=\{(1,3),(3,1),(4)\}$

$SP(5,3)=\{(2,3),(3,2),(5)\}$

$SP(6,3)=\{(6)\}$

$SP(7,3)=\{(1,6),(6,1),(4,3),(3,4),(7)\}$

$SP(8,3)=\{(2,6),(6,2),(5,3),(3,5),(8)\}$

$SP(9,3)=\{(9)\}$

$SP(10,3)=\{(1, 9), (9, 1), (4, 6), (6, 4), (7, 3), (3, 7), (10)\}$\\
Thus if we define $sp(n,m)=|SP(n,m)|$, we see that the following recurrence relation holds:

$sp(n,m)=0$ if $n<0$, $sp(0,m)=1$ and $sp(n,m)=1,\, 1\leq n\leq m-1$. Then for $n\geq m$,

\begin{equation}\label{eq1}
sp(n,m) = \begin{cases}
sp(n/m,m)& \mbox{if}\ n\equiv 0\pmod{m},\\
2\cdot sp(n-r,m)+sp(n-m,m)& \mbox{if}\ n\equiv r\pmod{m},0<r<m.
\end{cases}
\end{equation}
The case $m=2$ gives the function considered earlier: $sp(n,2)=sp(n)$.

We will associate the set of semi-$m$-Pell compositions with a class of restricted unimodal compositions into powers of $m$.
  
Let $oc(n,m)$ denote the number weakly unimodal $m$-power compositions of $n$ in which every part size occurs in one place with multiplicity not divisible by $m$. Thus for example, the following are some objects enumerated by 
$oc(92,3)$: $(27^2, 9^2, 3^2, 1^{14}), (1^8, 3^{13}, 9^2, 27)$ and $(3^{14}, 9, 27, 1^{14})$.

\begin{theorem}\label{thm1} For integers $n\geq 0,m>1$,
\begin{equation}\label{eq2}
sp(n,m) = oc(n,m),
\end{equation}
\end{theorem}
\begin{proof} First Proof (generating functions): Let $Q_m(x)=\sum_{n\geq 0}sp(n,m)x^n$. Then 
\begin{align*}
Q_m(x)&=\sum_{n\geq 0}sp(n,m)x^{nm} + \sum_{r=1}^{m-1}\sum_{n\geq 0}sp(nm+r,m)x^{nm+r}\\
&=\sum_{n\geq 0}sp(n,m)x^{nm} + \sum_{r=1}^{m-1}\sum_{n\geq 1}sp(nm+r,m)x^{nm+r} + \sum_{r=1}^{m-1}sp(r,m)x^{r}\\
&=Q_m(x^m) + \sum_{r=1}^{m-1}\sum_{n\geq 1}(2sp(nm,m)+sp(nm+r-m,m))x^{nm+r} + \sum_{r=1}^{m-1}x^{r}\\
&=Q_m(x^m) + 2\sum_{r=1}^{m-1}x^r\sum_{n\geq 1}sp(nm,m)x^{nm} + \sum_{r=1}^{m-1}\sum_{n\geq 0}sp(nm+r,m)x^{nm+r+m} + \sum_{r=1}^{m-1}x^{r}\\
&=Q_m(x^m) + 2(Q_m(x^m)-1)\sum_{r=1}^{m-1}x^{r} + \sum_{r=1}^{m-1}x^{r} + \sum_{r=1}^{m-1}\sum_{n\geq 0}sp(nm+r,m)x^{nm+r+m}.
\end{align*}

Eliminating the last sum by means of the first equality, we obtain 
$$Q_m(x) = Q_m(x^m) + 2(Q_m(x^m)-1)\sum_{r=1}^{m-1}x^{r} + \sum_{r=1}^{m-1}x^{r} +  x^m(Q_m(x)-Q_m(x^m)),$$
$$(1-x^m)Q_m(x) = Q_m(x^m) + 2Q_m(x^m)\sum_{r=1}^{m-1}x^{r} - \sum_{r=1}^{m-1}x^{r} - x^mQ_m(x^m)$$
which gives
$$Q_m(x) + \frac{\sum_{r=1}^{m-1}x^r}{1-x^m} = \frac{1+2\sum_{r=1}^{m-1}x^r-x^m}{1-x^m}Q_m(x^m).$$
To iterate the last equation, we begin by denoting both sides by $Q_{m,1}(x)$ and obtain
$$Q_{m,1}(x) = \frac{1+2\sum_{r=1}^{m-1}x^r-x^m}{1-x^m}\left(Q_{m,1}(x^m)-\frac{\sum_{r=1}^{m-1}x^{mr}}{1-x^{m^2}}\right)$$
which gives
$$Q_{m,1}(x) + \frac{(1+2\sum_{r=1}^{m-1}x^r-x^m)\sum_{r=1}^{m-1}x^{mr}}{(1-x^m)(1-x^{m^2})}=  \frac{1+2\sum_{r=1}^{m-1}x^r-x^m}{1-x^m}Q_{m,1}(x^m).$$
Denoting both sides of the last equation by $Q_{m,2}(x)$ we obtain
$$Q_{m,2}(x) = \frac{1+2\sum_{r=1}^{m-1}x^r-x^m}{1-x^m}\left(Q_{m,2}(x^m)-\frac{(1+2\sum_{r=1}^{m-1}x^{mr}-x^{m^2})\sum_{r=1}^{m-1}x^{m^2r}}{(1-x^{m^2})(1-x^{m^3})}\right)$$
which gives 
$$Q_{m,2}(x) +  \frac{(1+2\sum_{r=1}^{m-1}x^r-x^m)(1+2\sum_{r=1}^{m-1}x^{mr}-x^{m^2})\sum_{r=1}^{m-1}x^{m^2r}}{(1-x^m)(1-x^{m^2})(1-x^{m^3})}$$
$$=\frac{1+2\sum_{r=1}^{m-1}x^r-x^m}{1-x^m}Q_{m,2}(x^m).$$
In general we obtain, for any $i>0$,
$$Q_{m,i}(x) + \frac{x^{m^{i}r}}{1-x^{m^{i+1}}}\prod\limits_{t=0}^{i-1} \frac{1+2\sum_{r=1}^{m-1}x^{m^t r}-x^{m^{t+1}}}{1-x^{m^{t+1}}}
=\frac{1+2\sum_{r=1}^{m-1}x^r-x^m}{1-x^m}Q_{m,i}(x^m).$$
We suggest $Q_{m}(x)$ as the limiting value of the $Q_{m,i}(x)$ as $i$ tends to infinity, plus a 1 for our initial condition:
$Q_{m}(x)=1 + \lim\limits_{i\rightarrow\infty} Q_{m,i}(x)$.

$$Q_{m}(x) = 1 + \sum_{i=0}^\infty\frac{\sum_{r=1}^{m-1}x^{m^i r}}{1-x^{m^{i+1}}}\prod\limits_{t=0}^{i-1} \frac{1+2\sum_{r=1}^{m-1}x^{m^t r}-x^{m^{t+1}}}{1-x^{m^{t+1}}}.$$
We see that 
$$Q_{m}(x) = 1 + \sum_{i=0}^\infty\frac{\sum_{r=1}^{m-1}x^{m^i r}}{1-x^{m^{i+1}}}\prod\limits_{t=0}^{i-1}\left(1+ \frac{2\sum_{r=1}^{m-1}x^{m^t r}}{1-x^{m^{t+1}}}\right)=\sum_{n=0}^\infty oc(n,m)x^n.$$

As before, it is clear that this is the generating function for $oc(n,m)$, and it can be seen algebraically to satisfy the functional equation $$Q_m(x) = Q_m(x^m) + 2(Q_m(x^m)-1)\sum_{r=1}^{m-1}x^{r} + \sum_{r=1}^{m-1}x^{r} +  x^m(Q_m(x)-Q_m(x^m))$$ constructed earlier for $sp(n,m)$.

This completes the proof.
Some coefficients in the expansion of $Q_m(x)$ are displayed in Table \ref{tab1}.
\vskip 10pt

\ni Second Proof: We give a combinatorial proof. Let the sets enumerated by $sp(n,m)$ and $oc(n,m)$ be denoted by $SP(n,m)$ and $OC(n,m)$ respectively.

Each part $t$ of $C\in SP(n,m)$ can be expressed as $t=m^i\cdot h,\, i\geq 0$, where $m$ does not divide $h$. Now transform $t$ as follows:

$$t=m^i\cdot h\longmapsto m^i,m^i,\ldots,m^i\ (h\ \mbox{times}).$$

Note that the case $i=0$ may arise only as a first or last part of $C$ (by Theorem \ref{maxmy}). This gives a unique member of $OC(n,m)$ provided that we retain the clusters of the $m^i$, corresponding to each $t$, in consecutive positions, and maintain the order of the parts of the resulting $m$-power composition (as in the case of $m=2$).

To reverse the map we write each $\beta\in OC(n,m)$ in the on-place exponent notation, to get $\beta=(\beta_1^{u_1},\ldots,\beta_s^{u_s})$ with the $m\nmid u_i$, and containing at most one instance of a 1-cluster which may be $\beta_1^{u_1}$ or $\beta_s^{u_s}$. Since each $\beta_i$ has the form $2^{j_i},\, j_i\geq 0$, we apply the transformation:
$$\beta_i^{u_i}=(2^{j_i})^{u_i}\longmapsto 2^{j_i}u_i.$$
This gives a unique composition in $SP(n,m)$ provided that the resulting parts retain their relative positions. Indeed the image may contain at most one part $\equiv r\pmod{m}$ which occurs precisely when $j_i=0$.

We illustrate the bijection with $(14, 3, 18, 27)\in SP(62,3)$: 
$$(14, 3, 18, 27) = (3^0\cdot 14,3^1\cdot 1,3^2\cdot 2,3^3\cdot 1)\mapsto (1^{14},3,9^2,27)\in OC(62,3).$$
\end{proof}
\medskip

\ni We provide a full example with $n=13, m=3$, where $sp(13,3)=13=c(13,3)$. The following members of the respective sets correspond 1-to-1 under the bijective proof of Theorem \ref{thm1}: 
\vskip 5pt

$SP(13,3)$: $(1, 3, 9), (3, 9, 1), (1, 9, 3), (9, 3, 1), (1, 12), (12, 1), (4, 9), (9, 4), (7, 6), (6, 7),$ 

$\quad {} (10, 3), (3, 10), (13)$.
\vskip 4pt

$C(13,3)$: $(1, 3, 9), (3, 9, 1), (1, 9, 3), (9, 3, 1), (1, 3^4), (3^4, 1), (1^4, 9), (9, 1^4), (1^7, 3^2), (3^2, 1^7),$

$\quad {} (1^10, 3), (3, 1^{10}), (1^{13})$.\\

\begin{table}[h!]
\centering
\begin{tabular}{c|ccccccccccccccc}\hline
$n$ &1&2&3&4&5&6&7&8&9&10&11&12&13&14&15\\ \hline 
$m=2$ &1&1&3&1&5&3&11&1&13&5&23&3&29&11&51\\ \hline 
$m=3$ &1&1&1&3&3&1&5&5&1&7&7&3&13&13&3\\ \hline
$m=4$ &1&1&1&1&3&3&3&1&5&5&5&1&7&7&7\\ \hline
$m=5$ &1&1&1&1&1&3&3&3&3&1&5&5&5&5&1\\ \hline
$m=6$ &1&1&1&1&1&1&3&3&3&3&3&1&5&5&5\\ \hline
\end{tabular}
\caption{Values of $sp(n,m)$, for $2\leq m\leq 6,\, 1\leq n\leq 15$}\label{tab1}
\end{table}

\section{Structural Properties of Semi-$m$-Pell Compositions}\label{struct}
Following \cite{AMN} we define the \emph{max $m$-power} of an integer $N$ as the largest power of $m$ that divides $N$ (not just the exponent of the power). Thus using the notation $x_m(N)$, we find that $N=u\cdot m^s,\, s\geq 0$, where $m\nmid u$ and $x_m(N)=m^s$. So $x_m(N)>0$ for all $N$. 

For example, $x_2(50)=2$ and $x_5(216)=1$. 

We define three (reversible) operations on a composition  $C=(c_1,\ldots,c_k)$ with any $m>1$:

(i) If the first or last part of $C$ is less than $m$, delete it:
 
$\quad c_1<m\implies \tau_1(C)=(c_2,\ldots,c_{k})$ or $c_k<m\implies \tau_1(C)=(c_1,\ldots,c_{k-1})$;

(ii) If $m\nmid c_t>m$, then $\tau_2(c)=(c_1,\ldots,c_{t-1},c_{t}-m,c_{t+1},\ldots,c_k)$.

(iii) If $C$ consists of multiples of $m$, divide every part by $m$: $\tau_3(C)=(c_1/m,\ldots,c_k/m)$.\\
These operations are consistent with the recursive construction of the set $SF(n,m)$, where $\tau_3^{-1}, \tau_1^{-1}$ and $\tau_2^{-1}$ correspond, respectively, to the three quantities in the recurrence \eqref{eq1}.

\begin{lemma}\label{lemch1} Let $n>0,\, m>1$ be integers with $n\equiv r\pmod{m},\, 1\leq r<m$. 

If $C=(c_1,\ldots,c_k)\in SP(n,m)$, then $c_1\equiv r\pmod{m}$ or $c_k\equiv r\pmod{m}$. 
\end{lemma}
\begin{proof}
If $k\leq 2$, the assertion is clear. So assume that $k>2$ such that $c_i\equiv r\pmod{m}$ for a certain index $i\notin\{1,k\}$. Then we can apply $\tau_2$ several times to obtain the composition $\beta=\tau_2^s(C),\, s=\lfloor\frac{c_i}{m}\rfloor$, that contains $r$ which is neither in the first nor last position. But this contradicts the recursive construction of $\beta$. Hence the assertion holds for all $C\in SP(n,m)$.  
\end{proof}

\begin{remark}\label{rem1}
If Lemma \ref{lemch1} is violated when $k>2$, then the sequence of max $m$-powers of the parts of $C$ cannot be unimodal: if $m\nmid c_j$ with $c_j\notin \{c_1,c_k\}$, then $x_m(c_j)=1$.
\end{remark}

\begin{lemma}\label{lemch2} Let $H(n,m)$ denote the set of compositions $C$ of $n$ such that the sequence of max $m$-powers of the parts of $C$ are distinct and unimodal.
Then if $C\in H(n,m)$ and $\tau_i(C)\neq\emptyset$, then $\tau_i(C)\in H(N,m),\, i=1,2,3$, for some $N$.
\end{lemma}
\begin{proof} Let $C=(c_1,\ldots,c_k)\in H(n,m)$. If $C$ contains a part $r$ less than $m$, then $r=c_1$ or $r=c_k$ (by Lemma \ref{lemch1}). So $\tau_1(C)\in H(n-r,m)$ since the max $m$-powers remain distinct and unimodal. If $C$ contains a non-multiple of $m$, say $c_t>m$, then by Lemma \ref{lemch1}, $t\in\{1,k\}$. Therefore  $\tau_2(C)$, i.e., replacing $c_t$ with $c_t-m$, preserves the unimodality of $C$. So $\tau_2(C)\in H(n-m,m)$.   
Lastly, since the parts of $C$ have distinct max $m$-powers $\tau_3(C)=(c_1/m,\ldots,c_k/m)$ contains at most one non-multiple of $m$. Hence $\tau_3(C)\in H(n/m,m)$.
\end{proof}

We state an independent characterization of the Semi-$m$-Pell compositions.

\begin{theorem}\label{maxmy}
A composition $C$ of $n$ is a semi-$m$-Pell composition if and only if the sequence of max $m$-powers of the parts of $C$ are distinct and unimodal.
\end{theorem}

\begin{proof} We show that $SP(n,m)= H(n,m)$. Let $C=(c_1,\ldots,c_k)\in SP(n,m)$ such that $C\notin H(n,m)$. Denote the properties, 

P1: sequence of max $m$-powers of the parts of $C$ are distinct.

P2: sequence of max $m$-powers of the parts of $C$ is unimodal.\\
First assume that $C$ satisfies P2 but not P1. So there are $c_i > c_j$ such that $x_m(c_i)=x_m(c_j)$, and let $c_i=u_im^s,\, c_j=u_jm^s$ with $m\nmid u_i,u_j$. Observe that $\tau_1$ deletes a part less than $m$, if it exists, from a member of $H(v,m)$. So we can use repeated applications of $\tau_2$ to reduce a non-multiple modulo $m$, followed by $\tau_1$. This is tantamount to simply deleting the non-multiple of $m$, say $c_t$, to obtain a member of $H(N,m),\, N<v$, from Lemma \ref{lemch2}.
By thus successively deleting non-multiples from $C$, and applying $\tau_3^c,\, c>0$, we  
obtain a composition $E = (e_1,e_2,\ldots)$ with $e_i=v_im^w>e_j=v_jm^w$, where $m\nmid v_i,v_j$ and $w\leq s$. Then apply $\tau_3^w$ to obtain a composition $G$ with two non-multiples of $m$. Then by Lemma \ref{lem001}, $G\notin SP(n,m)$.
Secondly assume that $C$ satisfies P1 but not P2. Then by the proof of Lemma \ref{lemch1} and Remark \ref{rem1} $\tau_2^u(C)\notin SP(N,m)$ for some $u$.
Therefore $C\in SP(n,m)\implies C \in H(n,m)$.
   
Conversely let $C=(c_1,\ldots,c_k)\in H(n,m)$. If $C=(t),\, 1\leq t\leq m$, then $C\in SP(t,m)$. If $m|c_i$ for all $i$, then $\tau_3(C)= (c_1/m,\ldots,c_k/m)\in H(n/m,m)$ contains at most one part $\not\equiv 0\pmod{m}$, so $C\in SP(n,m)$. Lastly assume that $n\equiv r\not\equiv 0\pmod{m}$. Then $r\in C$ or $m<c_t\equiv r\pmod{m}$ for exactly one index $t\in\{1,k\}$. Thus $\tau_1(C)$ consists of multiples of $m$ while $\tau_2(C)$ still contains one part $\not\equiv 0\pmod{m}$. In either case $C\in SP(n,m)$. Hence $H(n,m)\subseteq SP(n,m)$. The the two sets are identical.
\end{proof}

As an illustration of Theorem \ref{maxmy} note, for example, that $(2,9,4),(1,4,2,8)\notin SP(15,2)$ because the sequence of max $m$-powers of the parts are not unimodal, and $(2,10,3),(3,4,6,2)\notin SP(15,2)$ because the max $m$-powers are not distinct.

\begin{theorem}\label{conjsolv1} Let $n,m$ be integers with $n\geq 0,\, m>1$. Then  
$$sp(nm+1,m)=sp(nm+2,m)=\cdots =sp(nm+m-1,m) = 1+2\sum\limits_{j = 1}^{n}sp(j,m).$$
\end{theorem}

\begin{proof}
We first establish all but the last equality. By definition, $sp(1,m)=sp(2,m)=\cdots sp(m-1,m)=1$, and since $sp(m,m)=1$, we have  $sp(m+1,m)=2sp(m,m)+sp(1,m)=3$. Similarly $sp(m+2,m)=3=sp(m+3,m)=\cdots =sp(2m-1,m)$.
Assume that the result holds for all integers $<nm$. 
Then with $1\leq r\leq m-1$ we have $sp(nm+r,m)=2sp(nm,m)+sp(nm-(m-r),m)$.
But $1\leq r\leq m-1\implies m-1\geq m-r\geq 1$ and $sp(nm-(m-r),m)$ is constant by the inductive hypothesis. Hence the result.

For the last equality, we iterate the recurrence \eqref{eq1}. For each $r\in [1,m-1]$,
\begin{align*}
sp(mv+r,m)&=2sp(mv,m)+sp(m(v-1)+r,m)\\
&=2sp(v,m)+2sp(v-1,m)+sp(m(v-2)+r,m)\\
&=\cdots\\
&=2sp(v,m)+2sp(v-1,m)+\cdots + 2sp(2,m)+sp(m+r,m)
\end{align*}
Since $sp(m+r,m)=2sp(m,m)+sp(r,m)=2sp(1,m)+sp(r,m)$, we obtain the desired result:

$sp(mv+r,m)=2sp(v,m)+2sp(v-1,m)+\cdots + 2sp(2,m)+2sp(1,m)+1$.
\end{proof}
\medskip

\begin{coro}\label{corospec} Given integers $m\geq 2$, then for any $j\geq 0$ and a fixed $v\in\{0,1,\ldots,m\}$,
$$sp(m^j(mv+r),m)=2v+1,\ 1\leq r\leq m-1.$$
\end{coro}
\begin{proof} By applying the first part of the recurrence \eqref{eq1} $j\geq 0$ times we obtain $sp(m^j(mv+r),m) = sp(mv + r,m)$. The last equality in Theorem \ref{conjsolv1} then gives
$$sp(mv + r, m)  = 1 + 2\sum\limits_{i = 1}^{v}sp(i,m),\, 0\leq v\leq m,\, 1\leq r <m.$$
We know from \eqref{eq1} that $sp(i,m)=1$ when $1\leq i\leq m-1$, and since $sp(m,m)=1$, we obtain $sp(mv + r, m) = 1 + 2\sum\limits_{i = 1}^{v}1 = 1 + 2v$. Lastly, when $v=0$, we have $sp(mv+r,m)=sp(r,m)=1+2(0)=1$, as expected.
\end{proof}

Corollary \ref{corospec} is a stronger version of Theorem \ref{conjsolv1} since the restriction of $v$ to the set $\{0,1,\ldots,m\}$ specifies a common value.

We note the interesting cases $v=0$ of Corollary \ref{corospec} below.
\begin{coro}\label{coro2} Given an integer $m\geq 2$, then
$$sp(m^ih,m)=1,\ 1\leq h\leq m-1,\, i\geq 0.$$
Hence
$$sp(m^i,m)=1,\ i\geq 0.$$
\end{coro}

\section{Arithmetic Properties}\label{pties} 
The following Lemma may be easily proved like Lemma \ref{lemm2}:
\begin{lemma}\label{lemm2a} $sp(n,m)$ is odd for all integers $n\geq 0,\, m>1$.
\end{lemma}

\ni We found that  
\vskip 4pt

(1a) $sp(6j+1,3)\equiv 1\pmod{4}\ \forall\ j>0,$
\vskip 5pt

(1b) $sp(6j+4,3)\equiv 3\pmod{4}\ \forall\ j>0;$
\vskip 5pt

(2a) $sp(8j+1,4)\equiv 1\pmod{4}\ \forall\ j>0,$
\vskip 5pt

(2b) $sp(8j+5,4)\equiv 3\pmod{4}\ \forall\ j>0;$
\vskip 5pt

$\ldots{} \ldots{} \ldots{} \ldots{} $.
\vskip 5pt

\ni The congruences (1a) to (3b) are special cases of the following pairs of infinite modulo 4 congruences.
\vskip 6pt

\begin{theorem}\label{congmm1} Given any integer $m\geq 2$, then

(i) $sp(2mj+1,m)\equiv 1\pmod{4}\ \forall\ j\geq 0$\\
and

(ii) $sp(2mj+m+1,m)\equiv 3\pmod{4}\ \forall\ j\geq 0.$
\end{theorem}
Note that Theorem \ref{congmm1} implies Theorem \ref{congm2}.

\begin{proof} By induction on $j$. Note that $sp(1,m)=1$ and $sp(m+1,m)=2sp(m,m)+sp(1)=3$. So the assertion holds for $j=0$. Assume that (i) and (ii) hold for all $r<j$.  

Then we first obtain 
\begin{align*}
sp(2mj+1,m) &= 2sp(2mj,m) + sp(2mj+1-m,m)\\
&= 2sp(2j,m) + sp(2m(j-1)+m+1,m).
\end{align*}
Then by Lemma \ref{lemm2a} $sp(2j,m)$ is odd, say $2u+1$, and the inductive hypothesis gives $sp(2m(j-1)+m+1,m)\equiv 3\pmod{4}$, say $4t+3$. Hence
$$sp(2mj+1,m) = 2(2u+1) + (4t+3) = 4(u+t)+5\equiv 1\pmod{4}$$
which proves part (i).
 
To prove part (ii) we have  
\begin{align*}
sp(2mj+m+1,m) &= 2sp(2mj+m,m) + sp(2mj+1,m)\\
&= 2sp(2j+1,m) + sp(2mj+1,m).
\end{align*}
Then by Lemma \ref{lemm2a} $sp(2j+1,m)$ is odd, say $2u+1$, and part (i) gives $sp(2mj+1,m)\equiv 1\pmod{4}$, say $4t+1$. Hence
$$sp(2mj+m+1,m) = 2(2u+1) + (4t+1) = 4(u+t)+3\equiv 3\pmod{4}.$$
\end{proof}

\ni We also found: 
\vskip 4pt

(1) $sp(16j+5,4)\equiv 0\pmod{3}\ \forall\ j\geq 0;$
\vskip 5pt

(2) $sp(49j+8,7)\equiv 0\pmod{3}\ \forall\ j\geq 0;$
\vskip 5pt

(3) $sp(100j+11,10)\equiv 0\pmod{3}\ \forall\ j\geq 0;$
\vskip 3pt

$\ldots{} \ldots{} \ldots{} \ldots{} $.
\vskip 3pt

\ni These three congruences are contained in the following infinite modulo 3 congruence.
\vskip 6pt

\begin{theorem}\label{congmm2} Given an integer $m\geq 4$ such that  $m\equiv 1\pmod{3}$, $1 \leq r < m$, then
\vskip 4pt

$sp(m^2j+m+r,m)\equiv 0\pmod{3}\ \forall\ j\geq 0$.
\end{theorem}

We give two proofs below.
\begin{proof}[First Proof of Theorem \ref{congmm2}]
We attack the problem from the $oc(n,m)$ characterization as one-place $m$-power compositions with multiplicities not divisible by $m$.  For convenience we will denote these $OC_m$ compositions.

Begin by noting one group of $OC_m$ compositions: valid compositions include $(1^{m^2j+m+r})$, $(1^r, m^{mj+1})$, and $(m^{mj+1},1^r)$, for three.

Next, consider all those compositions that include only $m$ (and no higher powers of $m$), and 1.  The number of $m$ in the composition determines the number of 1s.

Note that $mj$ is not a valid number of $m$ in the composition, since it is divisible by $m$; however, $mj-1$, $mj-2$, $\dots$, $mj-(m-1)$ are all valid numbers of $m$, and there is always a number of 1 congruent to $r$ mod $m$ with such choices.  Hence there are $2(m-1)$ such compositions, and since $m \equiv 1 \pmod{3}$, this collection numbers a multiple of 3.

Since $mj-m$ is not a valid number of $m$ in the composition, but $mj-m-1$, $\dots$, $mj-(2m-1)$ are, similar collections occur until we are down to one part of size $m$.

Thus compositions in which only parts of size 1 and $m$ occur contribute a multiple of 3 to the total number of compositions.

Now consider any valid choice of numbers and orderings of powers $m^2$, $m^3$, etc.  Suppose that these form a composition of $Cm^2$.  The remaining value to be composed is $m^2(j-C)+m+r$.  In particular, a number of 1s congruent to $r$ mod $m$ must be in the composition, and some number of $m^1$ ranging from 1 up to $m(j-C)+1$ will be in the composition.

For any valid choice of numbers and arrangement of the powers $m^i$ with $i \geq 2$, we now make a similar argument to the ``empty'' case before.  Group the six partitions in which there are no $m^1$ and the required 1s are on either side of composition, or the four possible arrangements in which there are $m(j-C)+1$ of $m^1$ and exactly $r$ of 1.  Of the other permissible numbers of $m^1$ and 1, there are four valid compositions for each, and there are a multiple of 3 such groups of compositions, as previously argued.

Thus the total number of $OC_m$ compositions of $m^2j+m+r$ is 0 mod 3, as claimed.
\end{proof}

This argument can likely be generalized to additional congruences.\\

In order to give a second proof of the theorem, we first prove a crucial lemma.
\begin{lemma}\label{lem3mod} 
If $m\equiv 1\pmod{3}$, then for any integer $j\ge 0$,  

$$\sum\limits_{i=1}^{mj+1} sp(i,m)\equiv 1\pmod{3}.$$
\end{lemma}
\begin{proof}
\begin{align*}\sum_{i = 1}^{mj+1}sp(i,m)=&\sum_{r=1}^{m-1}sp(r,m)+sp(m,m)+\sum_{r=1}^{m-1}sp(m+r,m)+sp(2m,m)\\
&+\sum_{r=1}^{m-1}sp(2m+r,m)+\cdots+\sum_{r=1}^{m-1}sp(m(j-1)+r,m)\\
&+sp(mj,m)+sp(mj+1,m)
\end{align*}
$$\hspace{1.7cm} =\sum_{t=1}^{j}sp(mt,m)+sp(mj+1,m) + \sum_{t=0}^{j-1}\sum_{r=1}^{m-1}sp(mt+r,m)$$
Then using Equation \eqref{eq1} and Theorem \ref{conjsolv1} we obtain 
\begin{align}\sum_{i = 1}^{mj+1}sp(i,m)&=\frac{sp(mj+1,m)-1}{2}+sp(mj+1,m) + \sum_{t=0}^{j-1}\sum_{r=1}^{m-1}sp(mt+r,m)\nonumber\\
&=\frac{1}{2}\left(3sp(mj+1,m)-1\right) + \sum_{t=0}^{j-1}\sum_{r=1}^{m-1}sp(mt+r,m).\label{mod31}
\end{align}
But
\begin{align*}E_j(m):&=\sum_{t=0}^{j-1}\sum_{r=1}^{m-1}sp(mt+r,m)=m-1+\sum_{t=1}^{j-1}\sum_{r=1}^{m-1}sp(mt+r,m)\\
&=m-1+\sum_{t=1}^{j-1}\left(2\sum_{r=1}^{m-1}sp(mt,m) + \sum_{r=1}^{m-1}sp(mt+r-m,m) \right)\ (\mbox{by Eq. \eqref{eq1}})\\
&=m-1+2(m-1)\sum_{t=1}^{j-1}sp(t,m) + \sum_{t=0}^{j-2}\sum_{r=1}^{m-1}sp(mt+r,m)\\
&=(m-1)sp(m(j-1),m) + \sum_{t=0}^{j-2}\sum_{r=1}^{m-1}sp(mt+r,m)\ (\mbox{by Th. \ref{conjsolv1}}).\\
\end{align*}
Therefore
\begin{equation}\label{mod32} E_j(m)=(m-1)sp(m(j-1),m) + E_{j-1}(m)\end{equation}
Iterating Equation \eqref{mod32} we obtain 
$$E_j(m)=(m-1)\sum_{h=1}^{u}sp(m(j-h),m) + E_{j-u}(m),\, 1\leq u\leq j-1.$$
In particular, the case $u=j-1$ with $E_1(m)=m-1\equiv 0\pmod{3}$, implies
$$E_j(m)\equiv 0\pmod 3.$$ 
Consequently, reducing Equation \eqref{mod31} modulo 3 gives 
$$\sum_{i = 1}^{mj+1}sp(i,m)\equiv \frac{1}{2}(0-1)+0\equiv 1\pmod{3}$$
which is the desired result.
\end{proof}

\begin{proof}[Second Proof of Theorem \ref{congmm2}]
By Theorem \ref{conjsolv1}, 
$$sp(m^2j+m+r,m) = sp(m(mj+1)+r,m) = 1+2 \sum_{i=1}^{mj+1}sp(i,m).$$
From Lemma \ref{lem3mod} the sum is congruent to $1$ modulo 3, say $3v+1$ for some $v$. Hence
$$sp(m^2j+m+r,m) = 1+2(3v+1)\equiv 0\pmod{3}.$$
This completes the proof.
\end{proof}

\section*{Acknowledgment} This research was supported by the Number Theory Focus Area Grant of DST-NRF Centre of Excellence in Mathematical and Statistical Sciences (CoE-MaSS) - Wits University.

\bibliographystyle{amsplain}

\end{document}